\documentclass[11pt,reqno]{amsart}
\usepackage[american]{babel}
\usepackage{amsmath}
\usepackage{dsfont,mathtools,amssymb}
\usepackage{color}
\mathtoolsset{showonlyrefs,showmanualtags}

\newtheorem{theorem}{Theorem}[section]
\newtheorem{proposition}[theorem]{Proposition}
\newtheorem{lemma}[theorem]{Lemma}
\newtheorem{corollary}[theorem]{Corollary}
\newtheorem{remark}[theorem]{Remark}

\numberwithin{equation}{section}
\numberwithin{figure}{section}

\newcommand{\e}{\varepsilon}

\renewcommand{\rho}{\varrho}
\renewcommand{\phi}{\varphi}

\newcommand{\IND}{{\bf 1}}

\DeclareMathOperator{\var}{Var}
\DeclareMathOperator{\cov}{Cov}

\DeclareMathOperator{\Ent}{Ent}

\newcommand{\grad}{\nabla}

\newcommand{\be}{\begin{equation}}

\newcommand{\cA}{\ensuremath{\mathcal A}} 
\newcommand{\cB}{\ensuremath{\mathcal B}}

\newcommand{\cE}{\ensuremath{\mathcal E}} 
\newcommand{\cF}{\ensuremath{\mathcal F}}

\newcommand{\cL}{\ensuremath{\mathcal L}}

\newcommand{\bbR}{{\ensuremath{\mathbb R}} }

\newcommand{\bbZ}{{\ensuremath{\mathbb Z}} }

\newcommand{\ent}{{\rm Ent} } 

\newcommand{\tc}{\, |\, }

\let\a=\alpha \let\b=\beta   \let\d=\delta  \let\e=\varepsilon
 \let\g=\gamma     \let\k=\kappa  \let\l=\lambda
            
\let\r=\rho      
  
\let\D=\Delta     \let\L=\Lambda 
\let\O=\Omega

\def\({\left(}
\def\){\right)}
\date{September 12, 2014}



\thanks{The authors acknowledge the support of the Simons Institute of Theory of Computing, Berkeley, for hosting them during Fall 2013, where collaboration on this work first began. The authors also thank the American Institute of Mathematics, Palo Alto, for its hospitality. This research is also supported in part by Tetali's NSF grant DMS-1101447.}
\title{Approximate tensorization of entropy at high temperature}
\author{Pietro Caputo}
\address{Pietro Caputo \\
Universit\`a  Roma Tre. 
}
\email{caputo@mat.uniroma3.it}
\author{Georg Menz}
\address{Georg Menz \\ Stanford University}
\email{gmenz@stanford.edu}
\author{Prasad Tetali}
\address{Prasad Tetali \\ 
Georgia Institute of Technology}
\email{tetali@math.gatech.edu}

\begin{document}
\begin{abstract}
We show that for weakly dependent random variables the relative entropy functional satisfies an approximate version of the standard tensorization property which holds in the independent case. As a corollary we obtain a family of dimensionless logarithmic Sobolev inequalities. In the context of spin systems on a graph, the weak dependence requirements resemble  the well known Dobrushin uniqueness conditions. Our results can be considered as a discrete counterpart of a recent work of Katalin Marton \cite{Marton}. We also discuss some natural generalizations such as approximate Shearer estimates and subadditivity of entropy.

\end{abstract}

\maketitle

\thispagestyle{empty}
\section{Introduction}

Consider a product measurable space $(\O,\cF)$ of the form
\begin{align}\label{o_prod}
(\O,\cF) = (\O_1,\cF_1) \times \cdots\times (\O_N,\cF_N)\,,
\end{align}
where $(\O_k,\cF_k)$, $k=1,\dots,N$ are given measurable spaces. Let $\mu$ be 
a probability measure on $(\O,\cF)$. When $\mu$ is a product $\mu = \otimes_{k=1}^N \mu_k$, with $\mu_k$ a probability measure on $(\O_k,\cF_k)$, then  
it is well known (see e.g.\ \cite{Book_Toulouse}) that the entropy functional satisfies the inequality
\begin{align}\label{tensor}
  \Ent_{\mu} (f) \leq \sum_k  \mu\left[\Ent_{\mu_k} (f)\right],
\end{align}  
for all bounded measurable functions $f:\O\mapsto [0,\infty)$. Here, as usual, $$\Ent_{\mu} (f) = \mu\left[f\log f\right] - \mu\left[f\right]\log \mu\left[f\right],$$ which equals $\mu[f]$ times the  relative entropy of $\nu = (f/\mu[f])\mu$ with respect to $\mu$.
We refer to inequality \eqref{tensor} as the tensorization property of entropy. 
In the general case where $\mu$ is not a product measure, we
define the probability measure $\mu_k$ by conditioning on all variables $x_j\in\O_j$, with $j\neq k$. Thus  $\mu_k[f]$ denotes the function given by \begin{align}\label{mukke}\mu_k[f](x) = 
\mu\left[f\,|\,x_j,\, j\neq k\right],\end{align}
and $\Ent_{\mu_k} (f)$ denotes the function $\mu_k[f\log f] - \mu_k\left[f\right]\log (\mu_k\left[f\right])$.
We shall investigate the validity of an approximate tensorization statement of the form
\begin{align}\label{app_tensor}
\Ent_{\mu} (f) \leq C\sum_k  \mu\left[\Ent_{\mu_k} (f)\right],
\end{align}  
for all bounded measurable functions $f:\O\mapsto [0,\infty)$, where $C>0$ is a constant independent of $f$. We say that $\mu$ satisfies $AT(C)$ whenever \eqref{app_tensor} holds. 
As we discuss below, if $\mu$ satisfies such a bound, then one can deduce entropy related functional inequalities such as log-Sobolev or modified log-Sobolev inequalities for the $N$-component systems as a consequence of the corresponding inequalities for each component. 

The idea that a system with weakly dependent components should display some kind of tensorization of entropy is implicitly at the heart of the large body of literature devoted to the proof of logarithmic Sobolev inequalities for spin systems satisfying Dobrushin's uniqueness conditions or more general spatial mixing conditions; see 
\cite{Zeg,StrZeg,Lu-Yau,MarOli,MarStFlour,GuiZeg,Cesi,PPP}. Perhaps surprisingly, none of these works addresses explicitly the validity of the statement \eqref{app_tensor}. 
Recently, the inequality \eqref{app_tensor} has been considered by Marton \cite{Marton} in the case of continuous spins, with $\O = \bbR^N$  and  $\mu$ an absolutely continuous measure of the form $\mu(dx)=e^{-V(x)}dx$. The author proves \eqref{app_tensor} under suitable weak dependence conditions that are formulated in terms of the Hessian of $V$. This allows her to derive the logarithmic Sobolev inequality beyond the usual Bakry-\'Emery criterion \cite{BakEme} or the more recent Otto-Reznikoff criterion \cite{OttoRez}. 

In this paper, we focus on deriving general sufficient conditions on $( \O,\cF,\mu)$ for inequality \eqref{app_tensor} to be satisfied. In particular, for spin systems with bounded local interactions, we shall establish that approximate tensorization holds as soon as the temperature 
is high enough, regardless of the single spin space and the underlying spatial structure. 

Next, we observe that the tensorization property \eqref{tensor} is a member of a much larger family of inequalities, often called Shearer inequalities, satisfied by product measures. In Section \ref{shearer1}
below we briefly discuss the problem of establishing approximate Shearer estimates for weakly dependent non-product measures. 

\subsection{Approximate tensorization and the Heat Bath chain}\label{sec:hb}
Before describing our results in detail, let us discuss some basic relations between approximate tensorization and functional inequalities for the 
{\em Heat Bath} Markov chain (also known as Glauber dynamics or Gibbs sampler). To define the latter, consider the operator $\cL$ given by 
\begin{align}\label{gibbs}
\cL f(x) = \sum_k (\mu_k[f](x) - f(x)), 
\end{align}  
where $f:\O\mapsto\bbR$. 
Then $\cL$ defines the infinitesimal generator of a continuous time Markov chain on $\O$, 
such that with rate $1$ independently each component $\O_k$, $k=1,\dots,N$ is updated by replacing $x_k$ with a value $x'_k$ sampled from the conditional distribution $ \mu\left[\cdot\,|\,x_j,\, j\neq k\right]$. The operator $\cL$ is a bounded self-adjoint operator in $L^2(\O,\mu)$ and the Markov chain is reversible with respect to $\mu$. We denote by $(e^{t\cL}, t\geq 0)$ the heat bath semigroup; see e.g.\ \cite{BookBaGeLe}. 
The Dirichlet form of the process is given by  
\begin{align}\label{dir_gibbs}
\cE(f,g)= \mu\left[f(-\cL g)\right] = \sum_k \mu\left[\cov_{\mu_k}(f,g)\right] ,
\end{align}  
where $f,g\in L^2(\O,\mu)$ and $\cov_{\mu_k}(f,g)$ denotes the covariance 
$$
\cov_{\mu_k}(f,g) =  \mu_k[f g] - \mu_k[f] \mu_k[g].  
$$
The following inequalities are commonly studied in the Markov chain literature.
Say that $\mu$ satisfies $P(C)$, or the Poincar\'e inequality with constant $C$, if 
 \begin{align}\label{gap_tensor}
\var_{\mu} (g) \leq C\sum_k  \mu\left[\var_{\mu_k} (g)\right],
\end{align}  
for any bounded function $g:\O\mapsto\bbR$, where $\var_{\mu}(g) = \mu[g^2] - \mu[g]^2$ denotes the variance.
Moreover, say that $\mu$ satisfies $LS(C)$, or the Log-Sobolev inequality with constant $C$, if 
 \begin{align}\label{ls_tensor}
\Ent_{\mu} (f) \leq C\sum_k  \mu\left[\var_{\mu_k} (\sqrt f)\right],
\end{align}  
for all bounded measurable functions $f:\O\mapsto [0,\infty)$. Finally, say that  $\mu$ satisfies $MLS(C)$, or the Modified Log-Sobolev inequality with constant $C$, if 
 \begin{align}\label{mls_tensor}
\Ent_{\mu} (f) \leq C\sum_k  \mu\left[\cov_{\mu_k} (f,\log f)\right],
\end{align}  
for all bounded measurable functions $f:\O\mapsto [0,\infty)$. Modified log-Sobolev inequalities have received increasing attention in recent years 
\cite{PPP,GQ,Goel,CP,BobTet}, also in view of their connections with mixing times of Markov chains  \cite{MonTet}.

 It is well known that $P(C)$ is equivalent 
to exponential decay of the variance in the form $\var_\mu(f_t)\leq e^{-2t/C}\var_\mu(f)$, for all $f\in L^2(\O,\mu)$ and for all $t\geq 0$, where $f_t=e^{t\cL}f$. Similarly, $MLS(C)$  is equivalent 
to exponential decay of the entropy in the form 
\begin{align}\label{hbmlsidec}
\ent_\mu(f_t)\leq e^{-t/C}\ent_\mu(f),
\end{align}  
for all $t\geq 0$, for all functions bounded measurable $f\geq 0$, while $LS(C)$ is equivalent to a hypercontractivity property
of the heat bath semigroup; see \cite{DS}. The following implications are also well known \cite{DS}: for any $C>0$, $LS(C)\Rightarrow MLS(C/4) \Rightarrow P(C/2)$. The approximate tensorization property $AT(C)$ is naturally linked to the above inequalities as summarized below. 
\begin{proposition}\label{athb}
The following implications hold for any $C>0$:
$$
AT(C)\Rightarrow P(C)\,,\quad LS(C)\Rightarrow AT(C)\Rightarrow MLS(C) 
$$
\end{proposition}
\begin{proof}
$AT(C)\Rightarrow P(C)$ follows by linearization: \eqref{gap_tensor} can be obtained  from \eqref{app_tensor} by considering functions $f$ of the form $1+\e g$ with $g$ bounded and taking the limit $\e\to 0$. In words, approximate tensorization of entropy implies approximate tensorization of variance, with the same constant $C$. To prove $AT(C)\Rightarrow MLS(C)$, observe that by Jensen's inequality, for all functions $f\geq 0$ and $k\in[N]$:
\begin{align}\label{hbk}
\ent _{\mu_k}(f)\leq \cov_{\mu_k}(f,\log f) .
\end{align}  
Finally, $LS(C)\Rightarrow AT(C)$ follows from the well known fact that $\var_{\mu}(\sqrt f)\leq \ent_\mu (f)$ for any probability measure $\mu$ and any bounded function $f\geq 0$.
\end{proof}

\section{Main results}
For simplicity of exposition we formulate our results in the case where each of the spaces $\O_k$ in \eqref{o_prod} is finite or at most countably infinite, but there is no difficulty in transferring the same proof e.g.\ to the case $\O=\bbR^N$. However, one should keep in mind that our main assumptions involve $L_\infty$ norms and therefore they are not ideally suited to deal with unbounded interactions.   

The 
weak dependence assumption is formulated as follows. For each $k\in[N] = \{1,\dots,N\}$, set
$\bar\O_k = \times_{j\in[N]:\,j\neq k}\O_j$ and write 
$\bar x_k\in\bar \O_k$ for the vector $(x_j,\,j\neq k)$. Similarly, for $i\neq k$, set
$\bar\O_{k,i} = \times_{j\in[N]:\,j\neq i, j\neq k}\O_j$ and write 
$\bar x_{k,i}\in\bar \O_{k,i}$ for the vector $(x_j,\,j\neq k, j\neq i)$.
For $x=(x_1,\dots,x_N)$ we write $x=(x_k,\bar x_k)$ 
and let
$$
\mu_k^{\bar x_k}(x_k)  = \mu(x_k\tc \bar x_k) = \frac{\mu(x_k,\bar x_k)}{\sum_{y_k\in\O_k}\mu(y_k,\bar x_k)}
$$
denote the conditional probability on $\O_k$, so that 
$$\mu_k[g](x)=\sum_{y_k\in \O_k} \mu_k^{\bar x_k}(y_k)g(y_k,\bar x_k),$$
for any  bounded $g:\O\mapsto\bbR$. 
For fixed $i\neq k$, consider the function $\phi_{i,k}:\O_i^2\times \O_k \times \bar\O_{k,i}\mapsto\bbR$ defined by 
 \begin{align}\label{couphi}
 \phi_{i,k}(x_i,y_i,x_k,\bar x_{i,k}) = \frac{\mu_k^{x_i,\bar x_{i,k}}(x_k)}{\mu_k^{y_i,\bar x_{i,k}}(x_k)}.
 \end{align}
Our main assumption is formulated in terms of the $\phi_{i,k}$ as follows. 
Define the coefficients
 \begin{gather}\label{coupal}
 \a_{i,k}= \sup_{x_i,y_i\in\O_i,x_k\in\O_k,\bar x_{i,k}\in\bar \O_{i,k} }\phi_{i,k}(x_i,y_i,x_k,\bar x_{i,k})\,,\quad 
 \\
 \d_{i,k}= \sup_{x_i,y_i\in\O_i,x_k,y_k\in\O_k,\bar x_{i,k}\in\bar \O_{i,k} }|\phi_{i,k}(x_i,y_i,x_k,\bar x_{i,k})-
 \phi_{i,k}(x_i,y_i,y_k,\bar x_{i,k})|.\nonumber 
 \end{gather}
Notice that if $\mu$ is a product measure then $\a_{i,k} = 1$ and $\d_{i,k}=0$ for all $i\neq k$. 

\begin{theorem}\label{mainth}
Suppose that the measure $\mu$ satisfies the ergodicity assumption \begin{align}\label{ergo}
\lim_{t\to\infty} \ent_\mu (f_t)  = 0\,, 
\end{align}
for every bounded measurable function $f:\O\mapsto\bbR_+$, where $f_t=e^{t\cL}f$, and $(e^{t\cL}, t\geq 0)$ is the heat bath semigroup. 
Assume that the coefficients $\{\a_{i,k},\d_{i,k}\}$ satisfy $\g+\k<1$ where 
$$
\g=\max_{i}\sum_{k:\,k\neq i}\big|\a_{i,k}\a_{k,i}- 1\big|\,,\quad 
\k=\frac14\max_{i}\sum_{k:\,k\neq i}  (\d_{k,i}+\d_{i,k}) \a_{i,k}\a_{k,i}. 
$$
Then the approximate tensorization \eqref{app_tensor} holds with $C= (1-\g-\k)^{-1}$.
\end{theorem}
The proof of Theorem~\ref{mainth} is given in Section~\ref{sec:pfth}.
As in Marton's paper \cite{Marton}, the proof follows the semigroup approach. An important difference in our argument 
is that we use the heat bath semigroup where Marton uses the Langevin diffusion. While the overall strategy of the proof is similar, our setting requires a different technique because of the lack of differential calculus. Moreover, in contrast with \cite{Marton}, we do not need to require a one-site log-Sobolev inequality in our assumptions. This allows us to establish the approximate tensorization for the invariant measure of  Markov chains 
without log-Sobolev inequality, or even without Poincar\'e inequality, see the comments after Corollary \ref{coro_coro} below.

\subsection{Applications}\label{apps}
Next, we discuss the implications of Theorem \ref{mainth} in specific examples. 
For ease of exposition we limit ourselves to probability measures of the following form. 
Let $\mu_0(x)$ denote a product measure on $\O$ of the form $\mu_0(x)=\prod_{i=1}^N\mu_{0,i}(x_i)$, where $\mu_{0,i}$ is a probability measure on $\O_i$ for each $i\in[N]$,  
and consider the probability measure $\mu$ on $\O$ given by 
 \begin{align}\label{intbado}
\mu(x) = \frac{\mu_0(x)\,e^{\,\b \,W(x)}}{Z}\,,\qquad W(x)=\frac12\sum_{i,j\in [N] } J_{i,j}\,w_{i,j}(x_i,x_j)\,,
\end{align}
where $Z$ is the normalizing factor, the coefficients $J_{i,j}\in\bbR$ are assumed to satisfy 
$J_{i,i}=0$, $J_{i,j}=J_{j,i}$, and we assume that the functions $w_{i,j}$ satisfy 
\begin{align}\label{intbadoo}
\|w_{i,j}\|_\infty = \sup_{x_i,x_j}|w_{i,j}(x_i,x_j)|<\infty.
\end{align}
Here $\b>0$ is a parameter, the inverse temperature,  measuring the strength of the interaction. At $\b=0$ there is no dependence and the inequality \eqref{app_tensor} holds with $C=1$.  
Notice that the function $W$  in \eqref{intbado} is bounded in the sense that
$$
\|W\|_\infty\leq \frac12\sum_{i,j\in [N] } |J_{i,j}|\,\|w_{i,j}\|_\infty <\infty,
$$
where the bound depends on $N$ in general.
We first observe that a simple perturbation argument can 
be applied to obtain approximate tensorization for any $\b>0$, with a constant $C$ depending on $\|W\|_\infty$ and $\b$. 
\begin{lemma}\label{simple}
Under the assumption \eqref{intbadoo} the measure $\mu$ in \eqref{intbado} satisfies the approximate tensorization \eqref{app_tensor} with $C=e^{6\b\|W\|_\infty}$. In particular, \eqref{hbmlsidec} holds with the same $C$.
\end{lemma}
\begin{proof}
One has 
$$
\ent_\mu(f)\leq  e^{2\b\|W\|_\infty}\ent_{\mu_0}(f) 
\leq  e^{2\b\|W\|_\infty} \sum_k  \mu_0\left[\Ent_{\mu_{0,k} }(f)\right],
$$
where the first bound follows from the Holley-Stroock  perturbation argument \cite{HS}, while the second one is \eqref{tensor}. One more application of the Holley-Stroock perturbation estimate yields $\Ent_{\mu_{0,k} }(f)\leq e^{2\b\|W\|_\infty}\Ent_{\mu_{k} }(f)$, and using $\mu_0(x)\leq e^{2\b\|W\|_\infty}\mu(x)$ one obtains the claim.
\end{proof}

The above lemma, using Proposition \ref{athb} and the estimate in \eqref{hbmlsidec}, shows in particular that the ergodicity assumption \eqref{ergo} is always satisfied
in this setting. However, it represents a very poor estimate unless $\|W\|_\infty$ does not depend on $N$. Below, we consider 
cases where the function $W$ is not bounded uniformly in $N$, including systems, such as the Ising model,  where  
a phase transition can occur by varying the parameter $\b$. The main corollary of Theorem \ref{mainth} is summarized as follows. 
\begin{corollary}\label{coro_coro}
Define $\e_{i,k}= 4\b | J_{k,i}|\|w_{i,k}\|_\infty$ and assume 
\begin{align}\label{hyp_coro}
q:=\max_i\sum_{k:\,k\neq i}e^{\e_{i,k}}(e^{2\e_{i,k}}-1) < \frac23.
 \end{align}
 Then, the measure $\mu$ in \eqref{intbado} 
 has the approximate tensorization \eqref{app_tensor} with $C = (1-\frac32 q)^{-1}$. 
 \end{corollary}
The proof of Corollary~\ref{coro_coro} is given in Section~\ref{sec:pfcor}. 
It is interesting to note that the estimate  of Corollary \ref{coro_coro} is uniform in the choice of the single probability distributions $\mu_{0,i}$ in \eqref{intbado}, since the smallness condition \eqref{hyp_coro} does not involve the single measures $\mu_{0,i}$. In particular, the single measures $\mu_{0,i}$ are not required to satisfy a Poincar\'e inequality or any other condition.  
Below, we discuss some specific applications of Corollary \ref{coro_coro}. For simplicity we limit ourselves to Glauber dynamics for discrete spin systems and interacting birth and death chains.

\subsection{Spin systems}\label{spins}
Consider the Ising model defined as follows. 
Let $\O= \{-1,+1\}^N$, and set
\begin{align}\label{gibbs_ising}
\mu(x) = \frac1Z\,
{\exp{\Big(\textstyle\frac12\b\sum_{i,j\in[N]} J_{i,j} x_ix_j + \sum_{i\in[N]} h_{i} x_i}\Big)}\,,
\end{align}  
where $Z$ is the normalizing factor, $\b>0$ is the inverse temperature,  
the couplings $J_{i,j}\in\bbR$ are assumed to satisfy $J_{i,j}=J_{j,i}$, and the $h_i\in\bbR$ are the so-called external fields. Since the external fields can be absorbed in the single measures $\mu_{0,i}$,  it is immediate to check that the above $\mu$ has the form \eqref{intbado} with $w_{i,j}(x_i,x_j)=x_ix_j$. 
 Therefore, uniformly in the external fields, the Ising model defined in \eqref{gibbs_ising}
 has the approximate tensorization \eqref{app_tensor} with $C = (1-\frac32 q)^{-1}$ as soon as \eqref{hyp_coro} holds.

A special case is the {\em ferromagnetic Ising model on a graph} $G=(V,E)$, $|V|=N$, which corresponds to the choice $J_{i,j}=\IND(\{i,j\}\in E)$.  In this case one can obtain the following explicit criterion. Let $\D = \max_i\sum_{k:\,k\neq i} |J_{ki}|$ denote the maximal degree of the graph. Using e.g.\ $e^{t}(e^{2t}-1)< 3t$ for $t< 1/5$ one finds that  
if $4\b< 1/5$, then $e^{\e_{i,k}}(e^{2\e_{i,k}}-1)< 12\b |J_{k,i}|$, so that \eqref{hyp_coro} is satisfied as soon as $\b\leq \b_0\D^{-1}$ with e.g.\ $\b_0=(18)^{-1}$. 
Another example is the mean field model or {\em Curie-Weiss model}, which corresponds to $J_{i,j} = \frac1N$ for all $i,j\in[N]$. In this case, reasoning as above one obtains that \eqref{hyp_coro} is satisfied as soon as e.g.\ $\b\leq \b_0=0.1$. The critical point of the Curie-Weiss model is at $\b=1$, and therefore it is well known that approximate tensorization cannot hold for $\b\geq 1$ since already the Poincar\'e inequality \eqref{gap_tensor} fails beyond this point; see  \cite{meanfieldIsing}. The above numbers $\b_0$ can be improved slightly by a more accurate analysis of the values of $\b$ which allow the estimate \eqref{hyp_coro}, but it is clear that they will generally be far from the optimal values.

The result of Corollary \ref{coro_coro} can actually be extended to a much larger class of spin systems, where the spin takes a finite number $s\geq 2$ of values. For example, letting $\O=\{1,\dots,s\}^N$ one may define the {\em Potts model} Gibbs measure
\begin{align}\label{gibbs_potts}
\mu(x) = \frac1Z\,
{\exp{\Big(\textstyle\frac12\b\sum_{i,j\in[N]} J_{i,j} \IND(x_i=x_j) + \sum_{i\in[N]} h_{i} x_i}\Big)}\,.
\end{align}  
With the same arguments of Corollary \ref{coro_coro}, one obtains, for example,  for the Potts model on a graph $G$ with maximal degree $\D$, that there exists $\b_0(s)>0$ such that the approximate tensorization \eqref{app_tensor} holds as soon as $\b\leq \b_0(s)\D^{-1}$ uniformly in the external fields. 

\begin{remark}\label{remzd}{\em 
We point out that in the case of spin systems on the lattice $\bbZ^d$, it is known that the Log-Sobolev inequality holds for the heat bath dynamics under so-called  ``strong spatial mixing" conditions; see \cite{StrZeg,MarOli,Cesi,PPP}. Moreover, it is known that these spatial mixing conditions can cover a larger region of the phase diagram than our Dobrushin condition \eqref{hyp_coro}; see \cite{MarOli}.  Our assumption  \eqref{hyp_coro} compares well with Zegarlinski's earlier result \cite{Zeg1}. 
Since $LS(C)$ implies $AT(C)$ by Proposition \ref{athb}, our results are weaker than already known estimates in these cases. The only interest here could be the very different nature of the proof. 
}
\end{remark}

\subsection{Interacting birth and death chains}\label{rw}
Here we investigate some special cases of the general model \eqref{intbado}  with unbounded variables. 
To fix ideas, consider the case where $\O = \bbZ_+^N$, where $\bbZ_+=\{0,1,\dots,\}$. 
Consider the probability measure $\mu$ on $\O$ given by \eqref{intbado}, where the $\mu_{0,i}$ are fixed reference probability measures on $\O_i=\bbZ_+$ defined as follows. 
Let $\nu(n)$, $n\in\bbZ_+$ denote a  probability measure
such that
\begin{align}\label{logcon}
\nu(n)^2\geq \frac{n+1}{n}\nu(n+1)\nu(n-1)\,,\qquad n\geq 1.
\end{align} 
Such a measure is called ultra log-concave; see e.g.\ \cite{Johnson}. The basic example is the Poisson distribution with parameter $\l>0$, with $\nu(n)=e^{-\l}\l^n/n!$. Let $F_i:\bbZ_+\mapsto\bbR$ denote
arbitrary functions such that $F_\infty:=\max_i\|F_i\|_\infty<\infty$ and define the probability 
 \begin{align}\label{logcon1}
\mu_{0,i}(x_i) = \frac{\nu(x_i)e^{F_i(x_i)}}{Z},
\end{align} 
where $Z$ denotes the normalization. From \cite[Theorem 3.1]{CDaiPos} we know that 
for each $i$, $\mu_{0,i}$ satisfies the following modified log-Sobolev inequality in $\bbZ_+$:
 \begin{align}\label{logcon2}
\ent_{\mu_{0,i}}(f) \leq C_0
\,\mu_{0,i}\left[\partial_if\partial _i\!\!\,\log f\right]\,,
\end{align} 
with $C_0=e^{4F_\infty}{\nu(1)/\nu(0)}$, where $f:\bbZ_+\mapsto\bbR_+$ and $\partial_if(x_i) := f(x_i+1)-f(x_i)$, $x_i\in\bbZ_+$. 
The inequality \eqref{logcon2} expresses the exponential decay of entropy for the birth and death process with birth rate $b(n)=1$ and with death rate $d(n)= \IND(n\geq 1)\mu_{0,i}(n-1)/\mu_{0,i}(n)$, see \cite{CDaiPos,ChafJoul}. The gradient operator $\partial_i$ is extended to functions $f$ on $\O$ by setting $\partial_i f(x) = f(x+e_i)-f(x)$, $e_i$ denoting the unit vector in the $i$-th direction.
\begin{corollary}\label{coro_int}
Consider the measure $\mu$ given by \eqref{intbado} with $\mu_{0,i}$ as above.
Suppose the interaction term $\b\, W$  satisfies the assumption \eqref{hyp_coro}. Then $\mu$ satisfies \eqref{app_tensor} with $C= (1-\frac32 q)^{-1}$. Moreover, one has the modified log-Sobolev inequality 
\begin{align}\label{cor_int1}
\ent_{\mu}(f) \leq K
\,\sum_{i}\mu\left[\partial_if\partial _i\!\!\,\log f\right]\,,
 \end{align}
  with constant $K=CC_0e^{1/3}$, for all $f:\O\mapsto\bbR_+$.
 \end{corollary}
The inequality \eqref{cor_int1}, which can be interpreted as the exponential decay of entropy for interacting birth and death processes, could have been established also by an extension of the discrete $\Gamma_2$ approach of \cite{CDaiPos}, see the recent paper \cite{DaiPos}; see also \cite{PPP} for an alternative approach. 
Let us remark that in contrast with the case of bounded spin systems it is essential here to consider the modified log-Sobolev inequality rather than the log-Sobolev inequality itself, since even the one-dimensional bound \eqref{logcon2} could fail if we replace the energy $\mu_{0,i}\left[\partial_if\partial _i\!\!\,\log f\right]$ by $\mu_{0,i}\left[(\partial _i\sqrt f)^2\right]$, as it is seen for example when $\nu$ is the Poisson distribution.
Finally, we point out that it would be desirable to prove a modified log-Sobolev inequality as in \eqref{cor_int1} under weaker assumptions than  \eqref{intbadoo}, in order to include unbounded interactions
of log-concave type, see \cite{CP} for some examples. This would be natural from a discrete $ \Gamma_2$ perspective; see \cite[Section 3.2]{DaiPos} where some progress in this direction was recently made in the case $N=2$.  
For continuous unbounded spins, the log-Sobolev inequality at high temperature, beyond the Bakry-\'Emery criterion, has been established in \cite{Zeg96,Yoshida,BodHel,OttoRez,Marton}.

\subsection{Approximate Shearer inequalities and subadditivity}\label{shearer1}
We conclude this introduction with some notes on possible extensions of the previous results.
Let $\cB$ be a cover of $[N]$, that is a family of subsets $B\subset[N]$ whose union equals $[N]$. 
Let $\deg_k(\cB)$ denote the degree of $k$ in $\cB$, that is the number of subsets $B\in\cB$ such that $B\ni k$, and set 
$$
n_-(\cB) = \min\{\deg_k(\cB)\,,\;k\in[N]\}\,,\quad n_+(\cB) = \max\{\deg_k(\cB)\,,\;k\in[N]\}.
$$ 
for the minimal and maximal degree, respectively. 
For any $B\subset [N]$, we write $$\mu_B=\mu(\cdot\tc x_j,\,j\in B^c),$$ for the conditional probability measure on $\O_i,\,i\in B$, obtained by conditioning $\mu$ on the value of all $x_j\in\O_j$, $j\notin B$. When $B=\{k\}$ for some $k$, then $\mu_B$ coincides with $\mu_k$ defined in \eqref{mukke}. Also, for any function $f$, we write $f_B = \mu[f\tc x_B]$, where $x_B= \{x_i,\;i\in B\}$. Note that, when $f$ is a probability density with respect to $\mu$, then $f_B$ is the density of the marginal of $f\mu$ on $x_B$ with respect to the marginal of $\mu $ on $x_B$. When $B=\{k\}$ we simply write $f_k$ for $f_{\{k\}}$.   
We recall that any probability measure $\mu$ satisfies the decomposition
 \begin{align}\label{decom}
  \ent_{\mu}(f)  =  \ent_{\mu}(f_B)  + \mu[\ent_{\mu_{B^c}}(f)].
\end{align}
We formulate the following version of Shearer-type  estimates.
\begin{proposition}\label{gen_shear}
For any product measure $\mu = \otimes_{k=1}^N \mu_k$, for any cover $\cB$, for any bounded measurable $f\geq 0$, 
\begin{align}\label{oshea21}
 \ent_{\mu}(f) \leq \frac1{n_-(\cB)}\sum_{B\in\cB} \mu[\ent_{\mu_{B}}(f)]\,.
 \end{align}
\end{proposition}
This bound can be derived from the classical Shearer estimate for Shannon entropy \cite{MadTet,BolloBal}. For the reader's convenience we give a proof of Proposition \ref{gen_shear} along these lines in Section \ref{shearer2}. Note that using \eqref{decom} one has that 
\eqref{oshea21} is equivalent to the inequality
\begin{align}\label{oshea2}
  \sum_{A\in\bar \cB} \ent_{\mu}(f_A)\leq n_+(\bar \cB)\, \ent_{\mu}(f),
\end{align}
where $\bar\cB$ denotes the complementary cover $\{A=[N]\setminus B, \;B\in\cB\}$. 

It is immediate to check that the tensorization statement \eqref{tensor} is the special case
of \eqref{oshea21} corresponding to the singleton cover $\cB=\cB_1:=\{\{k\}\,,\;k\in[N]\}$. Similarly, for the complementary cover $\cB=\cB_{N-1}:=\bar \cB_1$, \eqref{oshea2} reduces to the well known subadditivity property
of entropy for product measures:
 \begin{align}\label{subadda}
\sum_{k}\ent_\mu(f_k)\leq \ent_\mu(f).
\end{align}
In the case of non-product measures with weak dependences,  
it is natural to investigate the validity of an approximate  form of Proposition \ref{gen_shear} such as 
\begin{align}\label{app_s}
 \ent_{\mu}(f) \leq \frac{C(\cB)}{n_-(\cB)}\sum_{B\in\cB} \mu[\ent_{\mu_{B}}(f)]\,,
 \end{align}
 where $C(\cB)>0$ is a suitable constant. 
 Note that, in analogy with Proposition \ref{athb}, approximate Shearer estimates are naturally linked to Log-Sobolev inequalities and exponential decay of entropy for the block version of the heat bath chain with infinitesimal generator given by 
$$
\cL^\cB = \sum_{B\in\cB }(\mu_B - 1).
$$
The following is an immediate corollary of 
our main result Theorem \ref{mainth}. 
\begin{corollary}\label{coro_shear}
Suppose $\mu$ satisfies the assumptions of Theorem \ref{mainth} and let $C>0$ be the constant appearing in that theorem. 
Then for any cover $\cB$,  setting $\D(\cB):=  \max\{|B|\,,\; B\in\cB\}$, 
\begin{align}\label{app_sheak}
 \ent_{\mu}(f) \leq C\,\frac{\D(\cB)}{n_-(\cB)}\sum_{B\in\cB} \mu[\ent_{\mu_{B}}(f)]\,.
 \end{align}
\end{corollary}
We remark that \eqref{app_sheak} is far from optimal if $\D(\cB)$ is large, and it becomes useless if the maximal cardinality of $\cB$ grows linearly in $N$. In particular, it cannot be used to prove an approximate subadditivity statement
(corresponding to $\cB=\cB_{N-1}$, $\D(\cB)=N-1$) of the form 
  \begin{align}\label{app_subadda}
\sum_{k}\ent_\mu(f_k)\leq C\,\ent_\mu(f),
\end{align}
with a nontrivial constant $C>0$. 
The approximate subadditivity estimate \eqref{app_subadda} has been obtained with the constant $C=2$ in \cite{CLL,CLL2} when $\mu$ is  the uniform measure on the $N$-dimensional sphere or the uniform measure on the symmetric group of permutations $S_N$. While the value $C=2$ is sharp for the sphere \cite{CLL}, it remains open to find the optimal value of $C$ for the symmetric group. We are not aware of any result of that kind for e.g.\ high-temperature Ising systems. 
%
%
%
%

\section{Proof of Theorem \ref{mainth}}\label{sec:pfth}
Let $\cE(\cdot,\cdot)$ denote the Dirichlet form of the heat-bath chain discussed in Section \ref{sec:hb}.
Let $(e^{t\cL},\,t\geq 0)$ denote the heat-bath semigroup. For any bounded nonnegative function $f$ on $\O$, we write $f_t=e^{t\cL}f$ for its evolution at time $t$.
We need the following lemma. 
\begin{lemma}\label{derivate}
For every bounded $f\geq 0$, $k\in[N]$ one has 
\begin{align}\label{semi1}
\Ent_\mu(f) &= \int_0^\infty \cE(f_t,\log f_t)\, dt,
\\
\label{semi2}
\mu\left[\Ent_{\mu_k} (f)\right] &= \int_0^\infty \cE(f_t,\log (f_t/\mu_k[f_t]))\, dt.
\end{align}
\end{lemma}
\begin{proof}
From the assumption \eqref{ergo} one has $\ent_\mu(f_t)\to 0$ as $t\to\infty$. Therefore, 
to prove \eqref{semi1} it is sufficient to observe that 
$$
\frac{d}{dt}\,\Ent_\mu(f_t) = \mu[\cL f_t \log f_t]=-\cE(f_t,\log f_t),
$$
where we use $\frac{d}{dt}f_t=\cL f_t$ and  $\mu[f_t] = \mu[f]$ for all $t\geq 0$.

From Jensen's inequality one has $0\leq \mu\left[\Ent_{\mu_k} (f_t)\right]\leq \ent_\mu(f_t)$ and the latter tends to zero as $t\to\infty$ by \eqref{ergo}. Therefore, \eqref{semi2} follows from 
\begin{align}\label{semio2}
\frac{d}{dt}\,\mu\left[\Ent_{\mu_k} (f_t)\right] &= - \cE(f_t,\log (f_t/\mu_k[f_t]))\,.
\end{align}
To prove \eqref{semio2}, notice that
$$
\frac{d}{dt}\,\Ent_{\mu_k}(f_t) = \mu_k[\cL f_t \log f_t] - \mu_k[\cL f_t] \log \mu_k[f_t].
$$
Taking expectation with respect to $\mu$  one has
$$
\frac{d}{dt}\,\mu\left[\Ent_{\mu_k} (f_t)\right] =\mu[\cL f_t \log f_t] - \mu[\cL f_t \log \mu_k[f_t]] = - \cE(f_t,\log (f_t/\mu_k[f_t])).
$$
\end{proof}
Thanks to Lemma \ref{derivate}, in order to prove \eqref{app_tensor}, it is sufficient to prove that for all $f\geq 0$:
\begin{align}\label{semi_tensor}
\sum_k\cE(f,\log (f/\mu_k[f]))\geq \rho\, \cE(f,\log f),
\end{align}  
with $\rho=C^{-1}>0$.
Writing \eqref{dir_gibbs} explicitly one has
\begin{align}\label{e_fisher_information}
  \cE(f,g)&=    \frac12  
 \sum_{i=1}^N   \sum_{x\in\O}\sum_{y_i\in\O_i}\mu(x)\mu^{\bar x_i}_i(y_i)\grad_i f(x;y_i) \grad_i g(x;y_i)\,,
\end{align}
using the notation 
\begin{equation}\label{nota_grad}
\grad_i f(x;y_i) :=f(\bar x_i,y_i)-f(x).
\end{equation}
Then \eqref{semi_tensor} becomes
\begin{align}\label{q_tensor}
\sum_{k,i=1}^N
&\sum_x\sum_{y_i}\mu(x)\mu^{\bar x_i}_i(y_i)\grad_i f(x;y_i) \grad_i \big(\log f - \log \mu_k [f]\big)(x;y_i)\nonumber\\&
\quad \geq \rho \sum_{i=1}^N   \sum_x\sum_{y_i}\mu(x)\mu^{\bar x_i}_i(y_i)\grad_i f(x;y_i) \grad_i \log f(x;y_i).
\end{align}  
Noting that $\grad_k \log \mu_k [f](x;y_i)=0$ for all $k$, the left hand side in \eqref{q_tensor} satisfies
\begin{align} 
\sum_{k,i}
&\sum_x\sum_{y_i}\mu(x)\mu^{\bar x_i}_i(y_i)\grad_i f(x;y_i) \grad_i \big(\log f - \log \mu_k [f]\big)(x;y_i) \\
& =
 \sum_{i}\sum_x\sum_{y_i}\mu(x)\mu^{\bar x_i}_i(y_i)\grad_i f(x;y_i) \grad_i \log f (x;y_i)
\\ & \;\;\;+  \sum_{k,i:\;k\neq i}
\sum_x\sum_{y_i}\mu(x)\mu^{\bar x_i}_i(y_i)\grad_i f(x;y_i) \grad_i \log f(x;y_i)\\ & \;\;\; - 
\sum_{k,i:\;k\neq i}
\sum_x\sum_{y_i}\mu(x)\mu^{\bar x_i}_i(y_i)\grad_i f(x;y_i) \grad_i \log \mu_k [f](x;y_i)\,.
\label{e_decomp_Gibbs_beginning}
\end{align}
Let us consider the three terms appearing in the right hand side of \eqref{e_decomp_Gibbs_beginning}.
The first term is exactly what we have in the right hand side of \eqref{q_tensor} apart from the constant $\rho$. 
The essence of the argument below is to show that the last term is approximately compensated by the term preceding it, the correction being proportional to the first term with a proportionality constant that can be made tiny in the regime of weak interaction, so that \eqref{q_tensor} holds with some positive $\r$.  This program starts with a pointwise estimate of the term $ \grad_i \log \mu_k [f]$. 

\subsection{Estimate of $\grad_i \log \mu_k [f]$}
As in several related works (see e.g.\ \cite{Zeg,StrZeg,Lu-Yau,MarOli}), 
estimating gradients of functions of $\mu_k [f]$ yields a principal term (which will be responsible for the compensation in \eqref{e_decomp_Gibbs_beginning} alluded to above) plus a covariance term that needs to be suitably small.  
A new ingredient here with respect to these works is the use of the logarithmic mean $\L(a,b)$, 
defined as usual by \begin{align}\label{logmean} 
\L(a,b)=\frac{a-b}{\log a - \log b},\qquad a,b\geq 0,\,a\neq b,
\end{align} 
and $\L(a,a)=a$, for all $a\geq 0$; see however \cite{ErbMaas} for a recent extensive use of the logarithmic mean in the analysis of entropy decay. 

Fix $i\neq k$ and $x_i,y_i\in\O_i,\bar x_{k,i}\in\bar \O_{k,i}$. Introduce the probability measure $\nu_{k,i}^{x_i,y_i,\bar x_{k,i}}$ on $\O_k$ defined by
\begin{align}\label{alab3} 
\nu_{k,i}^{x_i,y_i,\bar x_{k,i}} (x_k) =
\frac{ \L\big(f(y_i,x_k,\bar x_{k,i}),f(x_i,x_k,\bar x_{k,i})\big) \mu_k^{x_i,\bar x_{k,i}}(x_k)}{\bar \nu_{k,i}^{x_i,y_i,\bar x_{k,i}}},
 \end{align}
where the normalization $  \bar \nu_{k,i}^{x_i,y_i,\bar x_{k,i}}$ is given by
$$
\bar \nu_{k,i}^{x_i,y_i,\bar x_{k,i}}= \sum_{x_k} \L\big(f(y_i,x_k,\bar x_{k,i}),f(x_i,x_k,\bar x_{k,i})\big) \mu_k^{x_i,\bar x_{k,i}} (x_k).
$$
For simplicity, we omit the dependence on $f$ in the notation \eqref{alab3}. Recall the definition 
\eqref{couphi} of $\phi_{i,k}$. The goal of this subsection is to establish the following estimate.
 \begin{proposition}\label{coro3}
 For every $k\neq i$, for all $x\in\O,y_i\in\O_i$:
\begin{align}\label{alab30} 
&|\grad_i \log \mu_k [f](x;y_i)| \leq
\a_{i,k}\sum_{x'_k} |\grad_i \log f(\bar x_k,x'_k;y_i) | \nu_{k,i}^{x_i,y_i,\bar x_{k,i}}  (x'_k) +
\\
& \quad\quad  +\a_{i,k}
\frac{\big|\cov_{\mu_k^{y_i,\bar x_{k,i}}}\big(f(y_i,\cdot,\bar x_{k,i}), \phi_{i,k}(x_i,y_i,\cdot,\bar x_{k,i})\big)\big|}{\big(\bar \nu_{k,i}^{x_i,y_i,\bar x_{k,i}}\sum_{x'_k} f(y_i,x'_k,\bar x_{k,i}) \mu_k^{y_i,\bar x_{k,i}}(x'_k)
\big)^{\frac12}}.
\end{align}
\end{proposition} 
\begin{proof}
We start with simple manipulations: for $k\neq i$,
\begin{align*} 
&\grad_i \log \mu_k [f](x;y_i) =  \log \mu_k [f](y_i,\bar x_{k,i}) -  \log \mu_k [f](x_i,\bar x_{k,i})\\
& \quad= \log \sum_{x_k}f(y_i,x_k,\bar x_{k,i}) \mu_k^{y_i,\bar x_{k,i}} (x_k) -  \log \sum_{x_k} f(y_i,x_k,\bar x_{k,i}) \mu_k^{x_i,\bar x_{k,i}} (x_k) \\ &\quad \quad + \log  \sum_{x_k} f(y_i,x_k,\bar x_{k,i}) \mu_k^{x_i,\bar x_{k,i}} (x_k) -  \log  \sum_{x_k} f(x_i,x_k,\bar x_{k,i}) \mu_k^{x_i,\bar x_{k,i}} (x_k).
\end{align*}
Using the function $\L$ in \eqref{logmean}, we have
\begin{align*} 
&\grad_i \log \mu_k [f](x;y_i) \\ &\quad = 
\frac{ \sum_{x_k}f(y_i,x_k,\bar x_{k,i}) \big(\mu_k^{y_i,\bar x_{k,i}}(x_k)- \mu_k^{x_i,\bar x_{k,i}} (x_k)\big)}{\L\big( \sum_{x_k} f(y_i,x_k,\bar x_{k,i}) \mu_k^{y_i,\bar x_{k,i}}(x_k) , \sum_{x_k} f(y_i,x_k,\bar x_{k,i}) \mu_k^{x_i,\bar x_{k,i}} (x_k)\big)} 
 \\ &\quad \quad \qquad+ 
 \frac{ \sum_{x_k} \grad_if(x;y_i) \mu_k^{x_i,\bar x_{k,i}}(x_k)} {\L\big( \sum_{x_k}f(y_i,x_k,\bar x_{k,i}) \mu_k^{x_i,\bar x_{k,i}}(x_k) , \sum_{x_k} f(x_i,x_k,\bar x_{k,i}) \mu_k^{x_i,\bar x_{k,i}} (x_k)\big)} 
 \end{align*}
 Note that 
 \begin{equation}\label{eq_monL}
 \L(a,b)\geq q\L(a',b')\,, \qquad \text{if}\;\;a\geq qa' \;\text{and}\; b\geq qb' \,,
 \end{equation}
 for $a',b',q\geq 0$.
This follows e.g.\ from the representation $\L(a,b)=\int_0^1 a^{1-t}b^t dt$ of the logarithmic mean \eqref{logmean}. Since $\L(a,a)=a$, one has
\begin{align}\label{1a_tensor}
 &\L\big(\textstyle\sum_{x_k}  f(y_i,x_k,\bar x_{k,i}) \mu_k^{y_i,\bar x_{k,i}}(x_k) ,\sum_{x_k} f(y_i,x_k,\bar x_{k,i}) \mu_k^{x_i,\bar x_{k,i}} (x_k)\big)\\&\qquad
\quad \geq\big[\sup_{x'_k}\phi_{i,k}(y_i,x_i,x'_k,\bar x_{k,i}) \big]^{-1}\textstyle\sum_{x_k} f(y_i,x_k,\bar x_{k,i}) \mu_k^{y_i,\bar x_{k,i}}(x_k)\\&\qquad\quad
\geq
(\a_{i,k})^{-1} \textstyle\sum_{x_k} f(y_i,x_k,\bar x_{k,i}) \mu_k^{y_i,\bar x_{k,i}}(x_k).
\end{align}  
Note
 that 
\begin{align*} 
&\sum_{x_k} f(y_i,x_k,\bar x_{k,i})\big(\mu_k^{y_i,\bar x_{k,i}}(x_k)- \mu_k^{x_i,\bar x_{k,i}} (x_k)\big) \\&
\qquad\qquad  = -\cov_{\mu_k^{y_i,\bar x_{k,i}}}\big(f(y_i,\cdot,\bar x_{k,i}), \phi_{i,k}(x_i,y_i,\cdot,\bar x_{k,i})\big).
\end{align*}
Moreover, the concavity of $(a,b)\mapsto \L(a,b)$ implies that
\begin{align}
&\L\big(\textstyle \sum_{x_k}f(y_i,x_k,\bar x_{k,i}) \mu_k^{x_i,\bar x_{k,i}}(x_k) , \sum_{x_k} f(x_i,x_k,\bar x_{k,i}) \mu_k^{x_i,\bar x_{k,i}} (x_k)\big)\nonumber \\ & 
\geq
\textstyle \sum_{x_k} \L\big(f(y_i,x_k,\bar x_{k,i}),f(x_i,x_k,\bar x_{k,i})\big) \mu_k^{x_i,\bar x_{k,i}} (x_k)=\bar \nu_{k,i}^{x_i,y_i,\bar x_{k,i}}.\label{3a_tensor}
\end{align}  
By definition \eqref{alab3} we can write:
\begin{align}\label{alab2} 
&\frac{ \sum_{x_k} |\grad_if(x;y_i) |\mu_k^{x_i,\bar x_{k,i}}(x_k)} {\bar \nu_{k,i}^{x_i,y_i,\bar x_{k,i}}
 }  = 
 \sum_{x_k}|\grad_i \log f(x;y_i) | \nu_{k,i}^{x_i,y_i,\bar x_{k,i}}  (x_k).
 \end{align}
Combining the above bounds we have obtained
\begin{align}\label{alab1} 
&|\grad_i \log \mu_k [f](x;y_i)| \leq \sum_{x_k}|\grad_i \log f(x;y_i) | \nu_{k,i}^{x_i,y_i,\bar x_{k,i}}  (x_k)
\nonumber \\
& \qquad +\a_{i,k}\,
\frac{\big|\cov_{\mu_k^{y_i,\bar x_{k,i}}}\big(f(y_i,\cdot,\bar x_{k,i}), \phi_{i,k}(x_i,y_i,\cdot,\bar x_{k,i})\big)\big|
}{\sum_{x_k} f(y_i,x_k,\bar x_{k,i}) \mu_k^{y_i,\bar x_{k,i}}(x_k)}. \end{align}
We now derive a slightly different bound on $\grad_i \log \mu_k [f](x;y_i)$. 
Namely,  \begin{align*}
      & \grad_i \log \mu_k [f](x;y_i) \\ & \quad = \frac{\sum_{x_k}f(y_i,x_k,\bar x_{k,i}) \mu_k^{y_i,\bar x_{k,i}}(x_k)-\sum_{x_k} f(x_i,x_k,\bar x_{k,i}) \mu_k^{x_i,\bar x_{k,i}} (x_k)}{\L\big(\sum_{x_k} f(y_i,x_k,\bar x_{k,i}) \mu_k^{y_i,\bar x_{k,i}}(x_k) ,\sum_{x_k} f(x_i,x_k,\bar x_{k,i}) \mu_k^{x_i,\bar x_{k,i}} (x_k)\big)} 
\\&
   \quad   = 
    \frac{-\cov_{\mu_k^{y_i,\bar x_{k,i}}}\big(f(y_i,\cdot,\bar x_{k,i}), \phi_{i,k}(x_i,y_i,\cdot,\bar x_{k,i})\big)}{\L\big(\sum_{x_k}  f(y_i,x_k,\bar x_{k,i}) \mu_k^{y_i,\bar x_{k,i}}(x_k) ,\sum_{x_k} f(x_i,x_k,\bar x_{k,i}) \mu_k^{x_i,\bar x_{k,i}} (x_k)\big)} \;
    \\
    & \qquad +\frac{\sum_{x_k}  \grad_if(x;y_i) \mu_k^{x_i,\bar x_{k,i}}(x_k)}{\L\big(\sum_{x_k}  f(y_i,x_k,\bar x_{k,i}) \mu_k^{y_i,\bar x_{k,i}}(x_k) ,\sum_{x_k} f(x_i,x_k,\bar x_{k,i}) \mu_k^{x_i,\bar x_{k,i}} (x_k)\big)}
    \end{align*}
    Using \eqref{eq_monL} and \eqref{3a_tensor}, we have
\begin{align*}
\L\Big(\sum_{x_k}  f(y_i,x_k,\bar x_{k,i}) \mu_k^{y_i,\bar x_{k,i}}(x_k) ,
 \sum_{x_k} f(x_i,x_k,\bar x_{k,i}) \mu_k^{x_i,\bar x_{k,i}} (x_k)\Big)\geq \frac{\bar \nu_{k,i}^{x_i,y_i,\bar x_{k,i}}}{\a_{i,k}}.
\end{align*}  
Therefore the first term in the expression of $ \grad_i \log \mu_k [f](x;y_i)$ above is bounded in absolute value by
   $$\frac{\a_{i,k}}{\bar \nu_{k,i}^{x_i,y_i,\bar x_{k,i}}}\,
   \big|\cov_{\mu_k^{y_i,\bar x_{k,i}}}\big(f(y_i,\cdot,\bar x_{k,i}), \phi(x_i,y_i,\cdot,\bar x_{k,i})\big)\big|
   .
$$
    Similarly, the second term is bounded by
    \begin{align*}
&  \frac{\a_{i,k}}{\bar \nu_{k,i}^{x_i,y_i,\bar x_{k,i}}}\,
\sum_{x_k} |\grad_if(x;y_i)| \mu_k^{x_i,\bar x_{k,i}}(x_k)\\
&\qquad = \a_{i,k}\sum_{x_k}|\grad_i \log f(x;y_i) | \nu_{k,i}^{x_i,y_i,\bar x_{k,i}}  (x_k).
    \end{align*}
    Thus, we have obtained the following estimate: 
\begin{align}\label{alab20} 
&|\grad_i \log \mu_k [f](x;y_i)| \leq \frac{\a_{i,k}}{\bar \nu_{k,i}^{x_i,y_i,\bar x_{k,i}}}\,
   \big|\cov_{\mu_k^{y_i,\bar x_{k,i}}}\big(f(y_i,\cdot,\bar x_{k,i}), \phi(x_i,y_i,\cdot,\bar x_{k,i})\big)\big|\nonumber\\
&\quad \qquad +\a_{i,k}\sum_{x_k}|\grad_i \log f(x;y_i) | \nu_{k,i}^{x_i,y_i,\bar x_{k,i}}  (x_k).
 \end{align}
Finally, using $\a_{i,k}\geq 1$ and putting together \eqref{alab1} and \eqref{alab20} 
it is immediate to obtain the desired bound \eqref{alab30}.
\end{proof}
The next task is to plug the bound of Proposition \ref{coro3} into the last term of \eqref{e_decomp_Gibbs_beginning}. This produces the two terms
\begin{align} 
&A:=
\sum_{k,i:\;k\neq i}\a_{i,k}
\sum_x\sum_{y_i}\mu(x)\mu^{\bar x_i}_i(y_i)|\grad_i f(x;y_i)| \times\nonumber \\
& \quad \qquad\times\sum_{x'_k} |\grad_i \log f(\bar x_k,x'_k;y_i) | \nu_{k,i}^{x_i,y_i,\bar x_{k,i}}  (x'_k),
\label{aterm}\\
&
B:=
\sum_{k,i:\;k\neq i}\a_{i,k}
\sum_x\sum_{y_i}\mu(x)\mu^{\bar x_i}_i(y_i)|\grad_i f(x;y_i)| \times\nonumber \\
& \quad \qquad\times\frac{\big|\cov_{\mu_k^{y_i,\bar x_{k,i}}}\big(f(y_i,\cdot,\bar x_{k,i}), \phi_{i,k}(x_i,y_i,\cdot,\bar x_{k,i})\big)\big|}{\big(\bar \nu_{k,i}^{x_i,y_i,\bar x_{k,i}}\sum_{x'_k} f(y_i,x'_k,\bar x_{k,i}) \mu_k^{y_i,\bar x_{k,i}}(x'_k)
\big)^{\frac12}}.\label{bterm}
\end{align}
Below, we analyze the two terms separately. We start with term $A$ which allows for the main cancellation in \eqref{e_decomp_Gibbs_beginning}. 
\subsection{The main cancellation}\label{main_canc}
Let us rewrite $A=\sum_{k,i:\;k\neq i} A_{k,i}$, with
\begin{align*} 
A_{k,i}&=\a_{i,k}
\sum_{\bar x_{k,i}}\sum_{x_i}\sum_{y_i}\mu(\bar x_{k,i},x_i)
\Big(\sum_{x_k}\mu_k^{x_i,\bar x_{k,i}}(x_k)\mu^{x_k,\bar x_{k,i}}_i(y_i)|\grad_i f(x;y_i)| \Big) \times \\ & 
\qquad\quad  \quad\times\Big(\sum_{x'_k} |\grad_i \log f(\bar x_k,x'_k;y_i) | \nu_{k,i}^{x_i,y_i,\bar x_{k,i}}  (x'_k)\Big),
 \end{align*}
where we use \begin{equation}\label{basic}
\mu(x)=\mu(\bar x_{k,i},x_i)\mu_k^{x_i,\bar x_{k,i}}(x_k)\,, \quad \mu(\bar x_{k,i},x_i) = \mu(\bar x_k) = \sum_{x_k}\mu(x).
\end{equation} 
Since
\begin{equation}\label{logmeanu}
\mu_k^{x_i,\bar x_{k,i}}(x_k)|\grad_i f(x;y_i)|
= \bar \nu_{k,i}^{x_i,y_i,\bar x_{k,i}}
 |\grad_i \log f(\bar x_k,x_k,y_i) | \nu_{k,i}^{x_i,y_i,\bar x_{k,i}}  (x_k),
\end{equation}
we obtain 
\begin{align*} 
A_{k,i}&\leq\a_{i,k}
  \sum_{\bar x_{k,i}}\sum_{x_i}\mu(\bar x_{k,i},x_i) \sum_{y_i}(\sup_{y_k}\mu_i^{y_k,\bar x_{k,i}}(y_i))\\ & \qquad \times \bar \nu_{k,i}^{x_i,y_i,\bar x_{k,i}}
 \Big(\sum_{x_k} |\grad_i \log f(\bar x_k,x_k,y_i) | \nu_{k,i}^{x_i,y_i,\bar x_{k,i}}  (x_k)\Big)^2
 \end{align*}
Using Schwarz' inequality one has
\begin{align*} 
A_{k,i}&\leq\a_{i,k}
  \sum_{\bar x_{k,i}}\sum_{x_i}\mu(\bar x_{k,i},x_i) \sum_{y_i}(\sup_{y_k}\mu_i^{y_k,\bar x_{k,i}}(y_i))\\ & \qquad \times \bar \nu_{k,i}^{x_i,y_i,\bar x_{k,i}}
 \sum_{x_k} |\grad_i \log f(\bar x_k,x_k,y_i) |^2 \nu_{k,i}^{x_i,y_i,\bar x_{k,i}}  (x_k))\\ & 
 \leq \a_{i,k}
  \sum_{\bar x_{k,i}}\sum_{x_i}\mu(\bar x_{k,i},x_i) \sum_{y_i}(\sup_{y_k,y'_k}\phi_{k,i}(y_k,y'_k,y_i,\bar x_{k,i}))\\ & \qquad \times \bar \nu_{k,i}^{x_i,y_i,\bar x_{k,i}}
 \sum_{x_k} |\grad_i \log f(\bar x_k,x_k,y_i) |^2 \mu_i^{x_k,\bar x_{k,i}}(y_i)\nu_{k,i}^{x_i,y_i,\bar x_{k,i}}  (x_k),
 \end{align*}
where we use 
$$
\frac{\sup_{y_k}\mu_i^{y_k,\bar x_{k,i}}(y_i)}{\inf_{y'_k}\mu_i^{y'_k,\bar x_{k,i}}(y_i)} = \sup_{y_k,y'_k}\phi_{k,i}(y_k,y'_k,y_i,\bar x_{k,i}).
$$
Therefore, using \eqref{logmeanu} and rearranging summations one arrives at 
\begin{align} \label{alab5} 
A_{k,i}&
 \leq \a_{i,k}\a_{k,i}\sum_x\sum_{y_i}\mu(x)\mu_i^{\bar x_{i}}(y_i)\grad_i f(x;y_i) \grad_i \log f(x;y_i).
 \end{align}
From \eqref{alab5} it follows that  \eqref{e_decomp_Gibbs_beginning}
can be bounded from below as follows:
\begin{align} 
\sum_{k,i}
&\sum_x\sum_{y_i}\mu(x)\mu^{\bar x_i}_i(y_i)\grad_i f(x;y_i) \grad_i \big(\log f - \log \mu_k [f]\big)(x;y_i) \\
& \geq -B + (1-\g)\sum_{i}\sum_x\sum_{y_i}\mu(x)\mu_i^{\bar x_i}(y_i)\grad_i f(x;y_i) \grad_i \log f (x;y_i) ,
\label{firststep}
\end{align}
where $B$ is given in \eqref{bterm} and 
\begin{align} \label{gammaco}
\g=\max_{i}\sum_{k:\,k\neq i}(\a_{i,k}\a_{k,i}-1).
\end{align}

\subsection{Covariance estimate}\label{cov_est}
The next step is an estimate of the form
\begin{align} 
&B\leq \k\sum_{i}\sum_x\sum_{y_i}\mu(x)\mu_i^{\bar x_{i}}(y_i)\grad_i f(x;y_i) \grad_i \log f (x;y_i),
\label{secstep}
\end{align}
for a suitable constant $\k>0$.
We start with the following statement. 
\begin{lemma}\label{covlemma}
For any fixed $k$ and for all configurations $\bar x_k\in\bar\O_k$, for any pair of functions $g,\psi:\O_k\mapsto\bbR$, with $g\geq 0$: 
\begin{align}\label{entrokk1} 
&
\big|\cov_{\mu_k^{\bar x_{k}}}\big(g, \psi\big)\big|\leq 
\frac 1{2} \,\big(\sup_{z_k,z'_k}|\psi(z_k)-\psi(z'_k)|\,\big)\Big(\sum_{x_k} g(x_k) \mu_k^{\bar x_k}(x_k)
\Big)^{\frac12}\times
\nonumber \\
& \qquad \times \Big(\sum_{x_k}\sum_{y_k}\mu_k^{\bar x_{k}}(x_k)\mu_k^{\bar x_{k}}(y_k)[g(y_k)-g(x_k)][\log g(y_k)-\log g(x_k)]
\Big)^{\frac12}.\end{align}
\end{lemma}
\noindent
\begin{proof}
Set 
$\g(x_k,y_k):= \mu_k^{\bar x_{k}}(x_k)\mu_k^{\bar x_{k}}(y_k)$, and 
write
$$
\cov_{\mu_k^{\bar x_{k}}}\big(g, \psi\big)=\frac12\sum_{x_k,y_k}\g(x_k,y_k)(g(y_k)-g(x_k))(\psi(y_k)-\psi(x_k)).
$$
Therefore, 
$$
\big|\cov_{\mu_k^{\bar x_{k}}}\big(g, \psi\big)\big|\leq\frac12\big(\sup_{z_k,z'_k}|\psi(z_k)-\psi(z'_k)|\,\big)
\sum_{x_k,y_k}\g(x_k,y_k)|g(y_k)-g(x_k)|.
$$
Schwarz' inequality yields
\begin{align*}
&\sum_{x_k,y_k}\g(x_k,y_k)|g(y_k)-g(x_k)|
\leq \Big(\sum_{x_k,y_k}\g(x_k,y_k)(g(x_k)+g(y_k))\Big)^{\frac12}\times\\ &
\qquad \times\Big(\sum_{x_k,y_k}\g(x_k,y_k)\frac{(g(x_k)-g(y_k))^2}{g(x_k)+g(y_k)}\Big)^{\frac12}\\ & \quad =\Big(2\sum_{x_k}g(x_k) \mu_k^{\bar x_k}(x_k)
\Big)^{\frac12}\Big(\sum_{x_k,y_k}\g(x_k,y_k)\frac{(g(x_k)-g(y_k))^2}{g(x_k)+g(y_k)}\Big)^{\frac12}
.\end{align*}
Since $\L(a,b)\leq (a+b)/2$ one has $$(a-b)^2/(a+ b)\leq \frac12(a-b)^2/\L(a,b)=\frac12 (a-b)(\log a - \log b).$$
This shows that
\begin{align*}
&
\sum_{x_k,y_k}\g(x_k,y_k)|g(y_k)-g(x_k)|
\leq\Big(\sum_{x_k} g(x_k) \mu_k^{\bar x_k}(x_k)
\Big)^{\frac12}\times \\ & \qquad \qquad \times \Big(\sum_{x_k,y_k}\g(x_k,y_k)[g(y_k)-g(x_k)][\log g(y_k)-\log g(x_k)]
\Big)^{\frac12}.
\end{align*}
\end{proof}
For all fixed $x_i,y_i,\bar x_{k,i}$, we apply Lemma \ref{covlemma} with $g(x_k) = f(y_i,x_k,\bar x_{k,i})$ and $\psi(x_k)=\phi_{i,k}(x_i,y_i,x_k,\bar x_{k,i})$. One finds 
\begin{align}\label{entrok1} 
&\Big|\cov_{\mu_k^{y_i,\bar x_{k,i}}}\Big(f(y_i,\cdot,\bar x_{k,i}), \phi_{i,k}(x_i,y_i,\cdot,\bar x_{k,i})\Big)\Big| 
\\ &\qquad \leq \frac12 \,\d_{i,k}\Big(\sum_{x_k} f(y_i,x_k,\bar x_{k,i}) \mu_k^{y_i,\bar x_{k,i}}(x_k)
\Big)^{\frac12}\big( \cB^{y_i,\bar x_{k,i}}_k(f,\log f)\big)^{\frac12},
\end{align}
where  $\d_{i,k}$ is defined in \eqref{coupal} and we use the notation 
\begin{align}\label{entrok101} 
\cB^{y_i,\bar x_{k,i}}_k(f,g)=
\sum_{x_k}\sum_{y_k}\mu_k^{y_i,\bar x_{k,i}}(x_k)\mu_k^{y_i,\bar x_{k,i}}(y_k)\grad_k f(\bar x_{i},y_i; y_k) \grad_k g(\bar x_{i},y_i; y_k).
\end{align}
Plugging this into \eqref{bterm} one has 
\begin{align*} 
&B\leq \frac12
\sum_{k,i:\;k\neq i}\a_{i,k}\d_{i,k}
\sum_x\sum_{y_i}\mu(x)\mu^{\bar x_i}_i(y_i)|\grad_i f(x;y_i)|\big(\bar \nu_{k,i}^{x_i,y_i,\bar x_{k,i}}\big)^{-1/2}\,\big( \cB^{y_i,\bar x_{k,i}}_k(f,\log f)\big)^{\frac12}.
\end{align*}
Using \eqref{logmeanu} and reasoning as in Section \ref{main_canc} one has
\begin{align*} 
B&\leq \frac12
\sum_{k,i:\;k\neq i}\a_{i,k}\d_{i,k}
\sum_{\bar x_{k,i}}\sum_{x_i}\mu(\bar x_{k,i},x_i)
\sum_{y_i}(\sup_{z_k}\mu_i^{z_k,\bar x_{k,i}}(y_i))\times \nonumber \\& 
\;\;\times \sum_{x_k}\nu_{k,i}^{x_i,y_i,\bar x_{k,i}}(x_k) |\grad_i \log f(x;y_i)| \big(\bar \nu_{k,i}^{x_i,y_i,\bar x_{k,i}}\big)^{1/2}
\big( \cB^{y_i,\bar x_{k,i}}_k(f,\log f)\big)^{\frac12}
\\ & \leq 
\frac12
\sum_{k,i:\;k\neq i}\a_{i,k}\d_{i,k}\sum_{\bar x_{k,i}}\sum_{x_i}\mu(\bar x_{k,i},x_i)
\sum_{y_i}(\sup_{z_k}\mu_i^{z_k,\bar x_{k,i}}(y_i))\times \nonumber \\& \;\;\times
\Big(\sum_{x_k}\nu_{k,i}^{x_i,y_i,\bar x_{k,i}}(x_k) |\grad_i \log f(x;y_i)|^2 \bar \nu_{k,i}^{x_i,y_i,\bar x_{k,i}}\Big)^{1/2}
\big( \cB^{y_i,\bar x_{k,i}}_k(f,\log f)\big)^{\frac12}.
\end{align*}
Using $ab\leq \frac12\,a^2 + \frac12\,b^2$ one has
\begin{align}\label{tterms} B &\leq 
\frac14\sum_{k,i:\;k\neq i}\a_{i,k}\d_{i,k}
\sum_{\bar x_{k,i}}\sum_{x_i}\mu(\bar x_{k,i},x_i)
\sum_{y_i}(\sup_{z_k}\mu_i^{z_k,\bar x_{k,i}}(y_i))\times \\& \quad\times
\sum_{x_k}\nu_{k,i}^{x_i,y_i,\bar x_{k,i}}(x_k) |\grad_i \log f(x;y_i)|^2 \bar \nu_{k,i}^{x_i,y_i,\bar x_{k,i}}+ \nonumber\\&
+ \frac14\sum_{k,i:\;k\neq i}\a_{i,k}\d_{i,k}\sum_{\bar x_{k,i}}\sum_{x_i}\mu(\bar x_{k,i},x_i)
\sum_{y_i}(\sup_{z_k}\mu_i^{z_k,\bar x_{k,i}}(y_i))\,
 \cB^{y_i,\bar x_{k,i}}_k(f,\log f).
\nonumber
\end{align}
Using again \eqref{logmeanu} and the coefficients $\a_{k,i}$ as in Section \ref{main_canc}, the first term in \eqref{tterms} is bounded
by 
\begin{align}
\label{fterm} &
\frac14\sum_{k,i:\;k\neq i}\a_{i,k}\a_{k,i}\d_{i,k}
\sum_{x}\mu(x)
\sum_{y_i}\mu_i^{x_k,\bar x_{k,i}}(y_i)
\grad_i f(x;y_i)\grad_i \log f(x;y_i)\end{align}
Next, we estimate the second term in \eqref{tterms}. Notice that $\cB^{y_i,\bar x_{k,i}}_k(f,\log f)$ depends on $y_i$ and not on  $x_i$. We are going to show that 
\begin{align}
\label{faterm} &
\sum_{x_i}\mu(\bar x_{k,i},x_i)(\sup_{z_k}\mu_i^{z_k,\bar x_{k,i}}(y_i))\leq \a_{k,i}\mu(\bar x_{k,i},y_i).
\end{align}
Write $ \mu(\bar x_{k,i})=\sum_{x_i}\mu(\bar x_{k,i},x_i) $ for the marginal on $\bar\O_{k,i}$.
Let $\mu_{i,k}^{\bar x_{k,i}}(y_i,x_k)$ denote the joint law at $(\O_i,\O_k)$ conditioned on $\bar x_{k,i}\in\bar\O_{k,i}$, and observe that 
$$
\sum_{x_k}\mu_{i,k}^{\bar x_{k,i}}(y_i,x_k) = \sum_{x_k}\mu_{i}^{x_k,\bar x_{k,i}}(y_i) \mu(x_k | \bar x_{k,i}) \geq \inf_{z_k\in\O_k}\mu_{i}^{z_k,\bar x_{k,i}}(y_i) .
$$
Therefore,
\begin{align*} 
&\sum_{x_i}\mu(\bar x_{k,i},x_i)
(\sup_{z_k}\mu_i^{z_k,\bar x_{k,i}}(y_i)) = \mu(\bar x_{k,i})\big(\sup_{z_k}\mu_i^{z_k,\bar x_{k,i}}(y_i)\big) 
\\
&\qquad 
\leq \a_{k,i}\,\mu(\bar x_{k,i})\big(\inf_{z_k}\mu_{i}^{z_k,\bar x_{k,i}}(y_i)\big) 
\leq \a_{k,i}\,\mu(\bar x_{k,i})\sum_{x_k}\mu_{i,k}^{\bar x_{k,i}}(y_i,x_k) \\ & \qquad 
= \a_{k,i}\sum_{x_k}\mu(\bar x_{k,i},y_i,x_k)= \a_{k,i}\mu(\bar x_{k,i},y_i).
\end{align*} 
This proves \eqref{faterm}. Thanks to this estimate, 
the second term in \eqref{tterms} is estimated with 
\begin{align} \label{sterm}&
\frac14\sum_{k,i:\;k\neq i}\a_{i,k}\a_{k,i}\d_{i,k}
\sum_{x}\mu(x)
\sum_{y_k}\mu_k^{x_i,\bar x_{k,i}}(y_k)
\grad_k f(x;y_k)\grad_k \log f(x;y_k).\end{align}
%
%
Thus, summing \eqref{fterm} and \eqref{sterm}, the final estimate on $B$ is given by \eqref{secstep} with the coefficient $\k$ defined by:
\begin{equation}\label{kappa} 
\k= \frac14\max_{i}\sum_{k:\,k\neq i}  \d_{k,i} \a_{i,k}\a_{k,i}+ \frac14\max_{k}\sum_{i:\,i\neq k} \d_{k,i} \a_{k,i}\a_{i,k} .
\end{equation}
This concludes the proof of Theorem \ref{mainth}.

%

\section{Proof of corollaries}\label{sec:pfcor}
\begin{proof}[Proof of Corollary \ref{coro_coro}]
From Lemma \ref{simple} it follows that \eqref{ergo} is satisfied. Define the function 
$$
\hat W_k(x)=\sum_{j:\,j\neq k} J_{jk} w_{jk}(x_j,x_k).
$$
Then, the measure in \eqref{intbado} satisfies
\begin{equation}\label{one-site}
\mu^{x_i,\bar x_{k,i}}_k(x_k) =
\frac{\mu_{0,k}(x_k)e^{\b \hat W_k(x)}}{\sum_{x'_k} \mu_{0,k}(x'_k)e^{\b \hat W_k(x'_k,\bar x_{k})}}.
\end{equation}
Notice that for $i\neq k$: $$
\hat W_k(\bar x_{k,i},y_i,x_k) = \hat W_k(\bar x_{k,i},x_i,x_k) + J_{k,i}(w_{i,k}(y_i,x_k)-w_{i,k}(x_i,x_k)).
$$
Setting $\e_{i,k} = 4\b |J_{k,i}|\|w_{i,k}\|_\infty$ it follows that
\begin{equation}\label{phik}
 e^{-\e_{i,k}} \leq
 \phi_{i,k}(x_i,y_i,x_k,\bar x_{k,i}) \leq e^{\e_{i,k}}.
\end{equation}
Moreover, from \eqref{phik} one has
$$
|\phi_{i,k}(x_i,y_i,x_k,\bar x_{i,k})-
 \phi_{i,k}(x_i,y_i,y_k,\bar x_{i,k})|\leq e^{\e_{i,k}} - e^{-\e_{i,k}}.
$$
Therefore, 
the coefficients $\a_{i,k}$ and $\d_{i,k}$ satisfy
\begin{equation}\label{aik} 
1\leq \a_{i,k} \leq e^{\e_{i,k}}\,, \qquad 0\leq \d_{i,k} \leq e^{\e_{i,k}} - e^{-\e_{i,k}}.
\end{equation}
The numbers $\g,\kappa$ in Theorem \ref{mainth} can then be bounded by 
$$
\g \leq \max_i\sum_{k\neq i}(e^{2\e_{i,k}}-1)\,,\qquad 
\kappa\leq \frac12 \max_i\sum_{k\neq i}e^{\e_{i,k}}(e^{2\e_{i,k}}-1) = \frac12 q.
$$ 
Under the assumptions of Corollary \ref{coro_coro} one has $\kappa < \frac13, \g \leq q < \frac23$, and therefore one may apply Theorem \ref{mainth} to obtain \eqref{app_tensor} with $C=(1-\g-\k)^{-1}\leq (1-\frac32q)^{-1}$.
\end{proof}

\begin{proof}[Proof of Corollary \ref{coro_int}]
From Corollary \ref{coro_coro} we know that \eqref{app_tensor} holds:
$$
\Ent_{\mu} (f) \leq C\sum_k  \mu\left[\Ent_{\mu_k} (f)\right].
$$ 
From \eqref{one-site} we also have: 
\begin{equation}\label{aiko} 
e^{-2\b\|\hat W_k\|_{\infty}}\leq \frac{\mu^{x_i,\bar x_{k,i}}_k(x_k)}{\mu_{0,k}(x_k)}  \leq e^{2\b\|\hat W_k\|_{\infty}}.
\end{equation}
Thus, the perturbation argument from Lemma \ref{simple} shows that 
$$
\Ent_{\mu} (f) \leq C\sum_k e^{2\b\|\hat W_k\|_{\infty}} \mu\left[\Ent_{\mu_{0,k}} (f)\right].
$$
At this point we can apply the bound \eqref{logcon2}. Therefore
$$
\Ent_{\mu} (f) \leq C'\sum_k e^{2\b\|\hat W_k\|_{\infty}} \mu\left[\mu_{0,k}\left[\partial_kf\partial _k\!\!\,\log f\right]\right],
$$
where $C'=CC_0$.
Using again \eqref{aiko}:
$$
\Ent_{\mu} (f) \leq C'\sum_k e^{4\b\|\hat W_k\|_{\infty}} \mu\left[\partial_kf\partial _k\!\!\,\log f\right].
$$
Finally, observe that 
$$
4\b\|\hat W_k\|_{\infty}\leq 4\b\sum_{j:\,j\neq k} |J_{j,k}| \|w_{j,k}\|_{\infty} = \sum_{j:\,j\neq k}\e_{j,k}\leq \frac{q}2,
$$
where we use $\e_{j,k}\leq \frac12(e^{2\e_{j,k}}-1)$ and $q$ is defined in Corollary \ref{coro_coro}. Therefore, $e^{4\b\|\hat W_k\|_{\infty}} \leq e^{q/2}\leq e^{1/3}$ and the conclusion \eqref{cor_int1} follows with $K = C' e^{1/3}$. 
\end{proof}

\section{Proof of Shearer-type estimates}\label{shearer2}
\begin{proof}[Proof of Proposition \ref{gen_shear}]
As usual, we restrict to the discrete setting. Suppose that $\cA$ is a uniform cover of $[N]$, namely a cover with constant degree, i.e.\ $n(\cA):=\deg_k(\cA)$ is independent of $k$. Let us start by showing that a product measure $\mu=\otimes_{i=1}^N\mu_i$ satisfies
 \begin{align}\label{shea2}
  \sum_{A\in\cA} \ent_{\mu}(f_A)\leq n(\cA) \ent_{\mu}(f). 
\end{align}
By homogeneity, we may assume 
$f$ to be a density w.r.t.\ $\mu$, i.e.\ $\mu[f]=1$. Call $X=(X_1,\dots,X_N)$ the random vector with values in the discrete space $\O$ whose probability distribution is $f\mu$. 
Then $f_A\mu_{A}$, where $ \mu_{A}:= \otimes_{i\in A}\mu_{i}$, is the law of the marginal $X_A=(X_i,\;i\in A)$.
The Shannon entropy $H(X_A)$ of $X_A$, for any $A\subset [N]$ satisfies:
\begin{align}\label{shannon}
H(X_A)&=-\sum_{x_A}f_A(x_A)\mu_{A}(x_A) \log(f_A(x_A)\mu_{A}(x_A)) \nonumber \\ & = -\ent_{\mu}(f_A) -
 \sum_{x_A}\sum_{i\in A}  f_A(x_A)\mu_{A}(x_A) \log (\mu_{i}(x_i))\nonumber \\
 & = -\ent_{\mu}(f_A) -
 \sum_{i\in A} \sum_{x_i} f_i(x_i)\mu_{i}(x_i) \log (\mu_{i}(x_i))\nonumber \\
&=  - \ent_{\mu}(f_A) + \sum_{i\in A} H(X_i) + \sum_{i\in A}\mu[f_i\log f_i].
\end{align}
In other words,
\begin{align}\label{shannon1}
\sum_{i\in A} H(X_i) -H(X_A)
= \ent_{\mu}(f_A) - \sum_{i\in A}\ent_{\mu}(f_i).
\end{align}
The standard Shearer estimate for Shannon entropy (see e.g.\ \cite{MadTet}) 
states that
\begin{align}\label{shear1}
n(\cA)\, H(X)\leq \sum_{A\in\cA} H(X_A).
\end{align}
Therefore, summing over $A\in\cA$ in \eqref{shannon1} and using \eqref{shear1}  
\begin{align}\label{shea3}
  &\sum_{A\in\cA} \ent_{\mu}(f_A) - n(\cA)\sum_{i\in[N]}\ent_{\mu}(f_i)
  \leq  n(\cA)\sum_{i\in[N]}H(X_i) - n(\cA) H(X).
\end{align}
Using \eqref{shannon1} with $A=[N]$ one obtains 
\eqref{shea2}.

Suppose now that $\cA$ is a non-uniform cover, i.e.\ $n_-(\cA)< n_+(\cA)$. 
We can add singleton sets to $\cA$ until we obtain a uniform cover $\cA'$ such that $n_+(\cA)=n(\cA')$.  It follows 
that 
$$
\sum_{A\in\cA} \ent_{\mu}(f_A)\leq \sum_{A\in\cA'} \ent_{\mu}(f_A)\leq 
n(\cA')\, \ent_{\mu}(f) = n_+(\cA) \,\ent_{\mu}(f). 
$$
This proves \eqref{oshea2}, which is equivalent to \eqref{oshea21}.
\end{proof}
\begin{proof}[Proof of Corollary \ref{coro_shear}]
From Theorem \ref{mainth} one has
\begin{align}\label{corea}
\ent_{\mu}(f)&\leq C\sum_k \mu\left[\ent_{\mu_k}(f)\right] \\
&= C\sum_k \sum_{B\in\cB:\; B\ni k}\frac{1}{\deg_k(\cB)}\,\mu\left[\ent_{\mu_k}(f)\right]\\
& \leq \frac{C}{n_-(\cB)} \sum_{B\in\cB}\sum_{k\in B}\mu\left[\ent_{\mu_k}(f)\right].
\end{align}
It remains to show that for any $B\subset [N]$:
$$
\sum_{k\in B}\mu\left[\ent_{\mu_k}(f)\right]\leq |B|\,\mu\left[\ent_{\mu_B}(f)\right].
$$
However, this is immediate since $\mu\left[\ent_{\mu_k}(f)\right]\leq \mu\left[\ent_{\mu_B}(f)\right]$, for any $k\in B$. 
\end{proof}

\bibliographystyle{plain}
\bibliography{ent}

\end{document}